\documentclass[11pt]{amsart}
\usepackage[all]{xy}
\usepackage{tikz-cd}
\usepackage{ifthen}
\usepackage{soul}
\usepackage{latexsym}
\usepackage{amsthm,amssymb}
\usepackage{amsbsy,amsfonts,amsmath}
\usepackage{enumitem}
\usepackage[ansinew]{inputenc}
\usepackage{graphicx,fancyhdr}
\usepackage{hyperref}
\usepackage{color}
\usepackage{mathtools}

\newtheorem{theorem}{Theorem}[section]

\newtheorem{coro}[theorem]{Corollary}

\newtheorem{prop}[theorem]{Proposition}

\theoremstyle{definition}
\newtheorem{definition}[theorem]{Definition}
\newtheorem{example}[theorem]{Example}

\theoremstyle{remark}
\newtheorem{remark}[theorem]{Remark}

\newcommand{\Rep}{\operatorname{Rep}}
\newcommand\Vc{\operatorname{Vec}}

\def\eps{\varepsilon}
\def\ot{\otimes}

\newcommand\Z{\mathbb Z}
\newcommand\N{\mathbb N}
\newcommand\I{\mathbb I}

\def\cB{\mathcal{B}}
\def\cE{\mathcal{E}}

\def\cM{\mathcal{M}}

\def\cC{\mathcal{C}}

\newcommand\op{\operatorname{op}}

\newcommand\Hom{\operatorname{Hom}}

\newcommand\id{\operatorname{id}}

\newcommand\co{\operatorname{co}}

\def\pf{\begin{proof}}
\def\epf{\end{proof}}
\newcommand{\yd}[1]{^{#1}_{#1}\mathcal{YD}}

\newcounter{commentcounter}

\begin{document}
\title[Pointed finite tensor categories over abelian groups]
{Pointed finite tensor categories over abelian groups}
\author[Angiono, Galindo]{Iv\'an Angiono and C\'esar Galindo}

\address{I. A.: FaMAF-CIEM (CONICET), Universidad Nacional de C\'ordoba,
Medina A\-llen\-de s/n, Ciudad Universitaria, 5000 C\' ordoba, Rep\'
ublica Argentina.}

\email{angiono@famaf.unc.edu.ar}

\address{C. G.: Departamento de matem\'aticas, Universidad de los Andes, Carrera 1 N. 18A -
10, Bogot\'a, Colombia.}

\email{cn.galindo1116@uniandes.edu.co}

\thanks{I. Angiono was partially supported by CONICET, Secyt (UNC), the
MathAmSud project GR2HOPF, and C. Galindo by Fondo de Investigaciones de la Facultad de
Ciencias de la Universidad de los Andes, Convocatoria 2017 para la
Financiaci\'on de proyectos de Investigaci\'on Categor\'ia: Profesores de Planta, proyecto ``Clasificaci\'on de algebras de Hopf de dimensi\'on peque\~{n}a". }

\begin{abstract}
We give a characterization of finite pointed tensor categories obtained as de-equivariantizations of finite-dimensional pointed Hopf algebras over abelian groups only in terms of the 
(cohomology class of the) associator of the pointed part. As an application we prove that every coradically graded pointed finite braided tensor category is a de-equivariantization of a finite dimensional pointed Hopf algebras over an abelian group.
\end{abstract}

\maketitle

\section{Introduction}

Let $\Bbbk$ be an algebraically closed field of characteristic zero. 
Via Tannakian reconstruction, tensor categories with a quasi-fiber functor with values in the category of finite dimensional $\Bbbk$-vector spaces correspond to categories $^H\cM$ of finite dimensional corepresentations
of a coquasi-Hopf algebra $H$ over $\Bbbk$.

A tensor category is pointed if every simple object is  invertible; this condition ensures the existence of a quasi-fiber functor as above. Hence any finite pointed tensor category  is equivalent to the category of comodules over a finite dimensional pointed coquasi-Hopf algebra.

In a previous paper \cite{AGP} written jointly with M. Pereira, we studied de-equivariantizations of Hopf algebras by applying Tannakian techniques.  We explicitly constructed  a coquasi-bialgebra such that its tensor category of comodules realizes the de-equivariantization of a Hopf algebra, \cite[Theorem 2.8]{AGP}. As an application, we defined a big family of pointed coquasi-Hopf algebras $A(H,G,\Phi)$ attached to a coradically graded pointed Hopf algebra $H$ and some extra group-theoretical data \cite[Proposition 3.3, Definition 3.5]{AGP}.

In this paper we pursue the study  of de-equivariantizations of Hopf algebras initiated
in \cite{AGP}. We characterize pointed finite tensor categories over abelian groups constructed as  de-equivariantizations of tensor categories of comodules over finite dimensional pointed Hopf algebras. For Hopf algebras, the de-equivariantization process generalizes the theory of central extensions of Hopf algebras. However, the central quotient is not necessarily a Hopf algebra but a coquasi-Hopf algebra. 

Given a tensor category $\cC$, we denote by $G(\cC)$ the group of isomorphism classes of invertible objects and by $\omega(\cC)\in H^3(G(\cC),\Bbbk^\times)$ the cohomology class defining the associator of the tensor subcategory of invertible objects.

Breen \cite[Proposition 4.1]{Breen} defined for every abelian group $\Lambda$ a  group homomorphism 
\begin{equation*}
\psi_\Lambda:H^3(\Lambda,\Bbbk^\times)\to \operatorname{Hom}(\wedge^3 \Lambda,\Bbbk^\times),
\end{equation*}
that measures if the category of Yetter-Drinfeld modules of $\Bbbk^\omega\Gamma$ (the coquasi-Hopf algebra defined over the group algebra $\Bbbk\Gamma$ with associator $\omega$) is pointed, see Theorem \ref{Condiciones equivalentes}.

A tensor category $\cC$ is \emph{coradically graded} if $\cC$ is equivalent to the category of comodules over a coradically graded coalgebra, see  \cite[Section 1.13]{Book-ENO} for a more categorical definition.
By \cite{angiono2016liftings} every finite-dimensional pointed Hopf algebra $H$ with abelian group of group-like elements $\Gamma$ is a cocycle deformation of $\cB(V )\#\Bbbk\Gamma$, where $V \in \yd{\Bbbk\Gamma}$ denotes the infinitesimal braiding of $H$ and $\cB(V)$ is the Nichols algebra of $V$. In particular $^H\cM$ and $^{\cB(V )\#\Bbbk\Gamma}\cM$ are tensor equivalent, so the pointed tensor categories obtained from $H$ or $\cB(V)\#\Bbbk\Gamma$ are the same. Therefore we may restrict just to coradically graded coquasi-Hopf algebras.
Our main result can be summarized as:
\begin{theorem}\label{main result}
A finite tensor category $\cC$ is tensor equivalent to a de-equivariantization of a pointed Hopf algebra over an abelian group if and only if $\cC$ is coradically graded, $G(\cC)$ is abelian and $\psi_{G(\cC)}(\omega_{\cC})\equiv 1$. Moreover, $\cC$ is realized as the corepresentations of a finite dimensional coquasi-Hopf algebra of the form $\cB(V)\#\Bbbk^{\omega_{\cC}}G(\cC)$.
\end{theorem}

Theorem \ref{main result} is proved in Section \ref{section:basic radically graded}, where a pointed Hopf algebra of the form $\cB(V)\#\Bbbk \Gamma$ is explicitly constructed. As a consequence of this result we obtain that every coradically graded pointed finite braided tensor category is tensor equivalent to a de-equivariantization of a coradically graded pointed Hopf algebra over an abelian group.

While the paper was at final stages of preparation, Huang, Liu, Yang and Ye posted the paper \cite{HLYYquasiHopf} containing results close to some of ours. They studied finite dimensional coquasi-Hopf algebras of the form $\cB(V)\# \Bbbk^\omega \Lambda$, with $\Lambda$ abelian and  $\omega$ trivializable. Their ideas for the construction of coquasi-Hopf algebras are different to our techniques, mainly they do not use the concept of de-equivariantization.
Instead, they find an specific representative for the cohomology class of a trivializable 3-cocycle and an specific trivialization, suitable for the construction of a pointed Hopf algebra of the form $H:=\cB(V)\# \Bbbk \Gamma$ such that the coquasi-Hopf algebra $\cB(V)\# \Bbbk^\omega \Lambda$ is a quotient of $H$. These examples correspond to the general construction in \cite{AGP}.

The organization of the paper is as follows. Section \ref{section:preliminaries} is devoted to preliminaries.
In Section \ref{trivializacion} we define the map $\psi_\Lambda:H^3(\Lambda,\Bbbk^\times)\to \operatorname{Hom}(\wedge^3 \Lambda,\Bbbk^\times)$ and characterizations of the condition  $\psi_\Lambda(\omega)=1$, which are used in the sequel. In Section \ref{section:basic radically graded} we prove generation in degree one for  coradically graded coquasi-Hopf algebras $A$ with 
associator $\omega \in H^3(G(A),\Bbbk^\times)$ such that $\psi_{G(A)}(\omega)=0$;
hence we prove \cite[Conjecture 1.32.1]{Book-ENO}  under the previous condition on the cohomology class determined by the associator. We prove Theorem \ref{main result}. We finish the section with an example of a coradically graded coquasi-Hopf algebra over $\Lambda=(\mathbb{Z}/2\mathbb{Z})^{\oplus3}$ with associator $\omega \in Z^3(\Lambda,\Bbbk^\times)$, such that $\psi_{\Lambda}(\omega)\neq 1$.

\subsection*{Notation}

Throughout the paper 
algebras and coalgebras are always defined over $\Bbbk$. We use Sweedler's notation for coalgebras omitting the sum symbol: $\Delta(c)= c_1\otimes c_2$ for all $c\in C$, $(C,\Delta,\varepsilon)$ a coalgebra.
Given a group $\Gamma$, $\widehat{\Gamma}$ denotes the group of characters of $\Gamma$ over $\Bbbk$, and $\langle \cdot,\cdot \rangle:\widehat{\Gamma}\times\Gamma\to\Bbbk^\times$ is the evaluation map.
For each $\theta \in \N_0$, we call $\I_{\theta} = \{n\in \N: n \le \theta \}$, or simply $\I$ if $\theta$ is clear from the context.
Also, $\delta_{x,y}$ is the Kronecker delta.

By a \emph{tensor category} we mean a $\Bbbk$-linear abelian category $\cC$ with finite dimensional $\Hom$ spaces and objects of finite length, endowed with a rigid $\Bbbk$-bilinear monoidal structure and such the unit object $\mathbf{1}$ is simple \cite{Book-ENO}. A tensor category is \emph{finite} if it is $\Bbbk$-linearly equivalent to the category of finite
dimensional  comodules over a finite dimensional $\Bbbk$-coalgebra.

\section{Preliminaries}\label{section:preliminaries}

In this section we recall some definitions and results about  coquasi-Hopf
algebras and  tensor categories.

\subsection{Coquasi-bialgebras}

A \emph{coquasi-bialgebra} $(H,m,u,\omega ,\Delta ,\varepsilon )$ is
a coalgebra $(H,\Delta ,\varepsilon )$ together with
coalgebra morphisms:
\begin{itemize}
\item the multiplication $m:H\otimes H\longrightarrow
H$ (denoted $ m(g\otimes h)=gh$),
\item the unit $u:\Bbbk \longrightarrow H$ (where we call
$ u(1)=1_{H} $),
\end{itemize}
and a convolution invertible element $\Omega \in (H\otimes H\otimes
H)^{\ast }$ such that
\begin{eqnarray}
h_{1}(g_{1}k_{1})\Omega (h_{2},g_{2},k_{2}) &=&\Omega
(h_{1},g_{1},k_{1})(h_{2}g_{2})k_{2},  \label{asociat multipl}
\\ 1_{H}h &=&h1_{H}=h,
\\ \Omega(h_{1}g_{1},k_{1},l_{1})\Omega (h_{2},g_{2},k_{2}l_{2}) &=& \Omega
(h_{1},g_{1},k_{1})\label{cocycle omega}
\\ && \qquad \times \Omega(h_{2},g_{2}k_{2},l_{1}) \Omega (g_{3},k_{3},l_{2}), \nonumber
\\ \Omega (h,1_{H},g) &=&\varepsilon (h)\varepsilon (g), \end{eqnarray}
for all $h,g,k,l\in H$. Note that 
\begin{align*}
\Omega (1_{H},h,g)& =\Omega (h,g,1_{H})=\varepsilon
(h)\varepsilon (g) & \text{for all } &g,h\in H.
\end{align*}
A coquasi-bialgebra $H$ is a \emph{coquasi-Hopf algebra} if there is a coalgebra map $\mathcal{S}:H\longrightarrow H^{\op}$
(the \emph{antipode}) and elements $\alpha $, $\beta \in H^{\ast }$ such that
\begin{align} \alpha (h)1_{H} &= \mathcal{S}(h_{1})\alpha (h_{2})h_{3},  \label{SalfaId}
\\ \beta (h)1_{H} &= h_{1}\beta (h_{2})\mathcal{S}(h_{3}),  \label{IdbetaS}
\\ \label{omega anihileaza S} \varepsilon (h)&=\omega (h_{1}\beta (h_{2}),\mathcal{S}(h_{3}),\alpha (h_{4})h_{5})
\\ &=\omega ^{-1}(\mathcal{S}(h_{1}),\alpha (h_{2})h_{3}\beta
(h_{4}),\mathcal{S}(h_{5})), & \text{for all }& h\in H.\nonumber
\end{align}

\begin{example}
Let $G$ be a discrete group. Recall that a (normalized)  3-cocycle $\omega \in Z^3(G, \Bbbk^\times)$ is a map $\omega:G\times G\to G\to \Bbbk^\times$ such that 
\begin{align*}
\omega(gh,c,l)\omega(g,h,kl)&=
\omega(g,h,k)\omega(g,hk,l)\omega(h,k,l), & \omega(g,1,h)=1,
\end{align*}
for all $g,h,k,l\in G$.  

Given $\omega \in Z^3(G,\Bbbk^\times)$, we define the coquasi-Hopf algebra $\Bbbk^{\omega} G$, with 
structure $(\Bbbk G, \Omega_\omega, S, \alpha,\beta)$, where $\Bbbk G$ is the group algebra with the usual comultiplication $\Delta(g)=g\otimes g$ for all $g\in G$, 
and $\Omega_\omega(g,h,k)=\omega(g,h,k)$ for all $g,h,k\in G$. The antipode structure is given by 
\begin{align*}
S(g)&=g^{-1}, & \alpha(g)&=1, & \beta(g)&=\omega(g,g^{-1},g)^{-1}, & 
\text{for all } &g\in G.
\end{align*}
\end{example}

\bigbreak

Let $H$ be a coquasi-Hopf algebra. The  category of left $H$-comodules $\ {}^{H} \mathcal{M}$ is rigid
and monoidal, where the tensor product is $\ot = \otimes_{\Bbbk}$, the comodule structure of the tensor product is the codiagonal one and the associator is
\begin{eqnarray*} \phi _{U,V,W} &:&(U\otimes V)\otimes
W\longrightarrow U\otimes (V\otimes W) \\ \phi _{U,V,W}((u\otimes
v)\otimes w) &=&\Omega (u_{-1},v_{-1},w_{-1})u_{0}\otimes
(v_{0}\otimes w_{0}) \end{eqnarray*} for $u\in U$, $v\in V$, $w\in
W$ and $U,V,W\in {}^{H}\mathcal{M}$. The dual coactions are given by
$\mathcal{S}$ and $\mathcal{S}^{-1}$, as in the case of Hopf
algebras.

\begin{example}
Let $G$ be a discrete group and $\omega \in Z^3(G,\Bbbk^\times)$. The tensor category $\ {}^{\Bbbk^{\omega}G} \mathcal{M}$ is
 $\Vc_G^\omega$, the category of $G$-graded vector spaces with associator induced by $\omega$.
\end{example}

\subsection{Braided tensor categories}
A tensor category $\cC$ is called \emph{braided} if it is endowed with a natural isomorphism 
\begin{align*}
c_{X,Y} & : X \otimes Y\to  Y \otimes X, & X,Y &\in\cC,
\end{align*}
satisfying the hexagon axioms, see \cite{JS}. 

\begin{example}[Pointed braided fusion categories]\label{Ex pointed braid}
Let $\cB$ be a pointed braided fusion category. The set of isormorphism classes of simple objects $\Gamma:=G(\cB)$ is an abelian group with product induced by the tensor product.

The associativity constraint defines a 3-cocycle $\omega \in Z^3(\Gamma,\Bbbk^\times)$. The braiding defines a function $c:\Gamma\times\Gamma\to \Bbbk^{\times}$ satisfying the following equations:
\begin{align}\label{eq:abelian-cocycle}
&
\begin{aligned}
\frac{c(g,hk)}{c(g,h)c(g,k)}&=\frac{\omega(g,h,k)\omega(h,k,g)}{\omega(h,g,k)}\\
\frac{c(gh,k)}{c(g,k)c(h,k)}&=\frac{\omega(g,k,h)}{\omega(g,h,k)\omega(k,g,h)},
\end{aligned}
& 
\text{for all } & g,h,k \in \Gamma.
\end{align}
These equations come from the hexagon axioms.
A pair $(\omega,c)$ satisfying \eqref{eq:abelian-cocycle} is called an \emph{abelian 3-cocycle}.
Following \cite{EM1,EM2} we denote by $Z^3_{ab}(\Gamma, \Bbbk^\times)$ the abelian group of all abelian 3-cocycles $(\omega,c)$.

An abelian 3-cocycle $(\omega,c)\in Z_{ab}^3(\Lambda,\Bbbk^\times)$ is called an \emph{abelian 3-coboundary} if there is $\alpha:\Lambda^{\times 2}\to \Bbbk^\times$, such that
\begin{align}\label{eq:abelian-coboundary}
&
\begin{aligned}
\omega(g,h,k)&=\frac{\alpha(g,h)\alpha(gh,k)}{\alpha(g,hk)\alpha(h,k)}\\
c(g,h)&=\frac{\alpha(g,h)}{\alpha(h,g)},
\end{aligned}
& 
\text{for all } & g,h,k \in \Lambda.
\end{align}
$B^3_{ab}(\Lambda,\Bbbk^\times)$ denotes the subgroup of $Z_{ab}^3(\Lambda,\Bbbk^\times)$ of abelian 3-coboundaries. The quotient group $H^3_{ab}(\Lambda,\Bbbk^\times):=Z_{ab}^3(\Lambda,\Bbbk^\times)/B^3_{ab}(\Lambda,\Bbbk^\times)$ is called the \emph{third group of abelian cohomology} of $\Lambda$.
\end{example}

\begin{example}[Corepresentations of coquasitriangular coquasi-Hopf algebras]

A \emph{coquasitriangular} coquasi-Hopf algebra is a pair $(H,r)$, where $H$ is a coquasi-Hopf algebra and $r:H\otimes H \to \Bbbk$ is  a convolution invertible map such that
\begin{align}
   &r(x_1,y_1)x_2y_2=y_1x_1r(y_2,x_2),\label{r-matrix 1}\\
   r(x,yz)&=\Omega(y_1,z_1,x_1)r(x_2,z_2)\Omega^{-1}(y_2,x_3,z_3)r(x_4,y_3)\Omega(x_5,y_4,z_4) \label{r-matrix 2}\\
    r(xy,z)&=\Omega(z_1,y_1,x_1)^{-1}r(x_2,z_2)\Omega(x_3,z_3,y_2)r(y_3,z_4)\Omega^{-1}(x_4,y_4,z_5),\label{r-matrix 3}
\end{align}
for all $x,y,z \in H$.
If $(H,r)$ is a coquasitriangular coquasi-Hopf algebra the $r$-form defines a braiding by
\begin{align*}
c_{V,W}:V\otimes W &\to W\otimes V, & 
v\otimes w&\mapsto r(v_{-1},w_{-1})w_0\otimes v_0.
\end{align*}Hence $(^H\cM,c)$ is a braided tensor category.

\end{example}
\begin{example}[Center construction]
An important example of a braided tensor category is the center $\mathcal{Z}(\cC)$ of a 
tensor category $(\cC,a,\mathbf{1})$. The center construction produces a braided tensor category
$\mathcal{Z}(\cC)$ from any tensor category $\cC$.
Objects of $\mathcal{Z}(\cC)$ are pairs $(Z, c_{-,Z})$, where $Z
\in \cC$ and $c_{-,Z} : - \otimes Z \to Z \otimes -$ is a natural
isomorphism such that the diagram

 \begin{equation}\label{equ: centro}
\xymatrix{&  Z\ot (X\ot Y)\ar[r]^{a_{Z,X,Y}}& (Z\ot X)\ot Y &\\ 
(X\ot Y)\ot Z \ar[ru]^{c_{X\ot Y,Z}}\ar[rd]_{a_{X,Y,Z}} &&& (X\ot Z)\ot Y \ar[lu]_{c_{X,Z}\ot \operatorname{id}_Y}\\
& X\ot( Y\ot Z) \ar[r]_{\operatorname{id}_X\ot c_{Y,Z}}&  X\ot (Z\ot Y) \ar[ru]_{a^{-1}_{X,Z,Y}}&
}
\end{equation}
commutes for all $X,Y,Z \in \cC$.
The braided tensor structure is the
following:
\begin{itemize}
\item the tensor product is $ (Y, c_{-,Y}) \otimes (Z, c_{-,Z}) =
(Y\otimes Z, c_{-,Y\otimes Z})$, where
\begin{align*}
c_{X,Y \otimes Z}& : X \otimes Y
\otimes Z \to Y \otimes Z \otimes X, \qquad X \in \cC,
\\
c_{X,Y \otimes Z}&= a_{Y,Z,X}(\id_Y \otimes c_{X,Z})a_{Y,X,Z}^{-1}(c_{X,Y} \otimes \id_Z)a_{X,Y,Z}.
\end{align*}
\item the braiding is the morphism $c_{X,Y}$,
\end{itemize}

\end{example}

\begin{example}[The Drinfeld center of $\Vc_\Lambda^\omega$.]
Let $\Lambda$ be a discrete group and $\omega\in Z^3(\Lambda,\Bbbk^\times)$. The Drinfeld center of $\Vc_\Lambda^\omega$ is equivalent to  $\yd{\Bbbk^{\omega}\Lambda}$,  the category  of Yetter-Drinfeld modules over $\Bbbk^{\omega}\Lambda$. The objects of $\yd{\Bbbk^{\omega}\Lambda}$ are $\Lambda$-graded vector spaces $V=\bigoplus_{g\in \Lambda} V_{g}$ with a linear map $\triangleright: \Bbbk^{\omega}\Lambda\ot V\to V$ such that $1\triangleright v=v$ for all $v\in V$,
\begin{align*}
(gh)\triangleright v &= \frac{\omega(g,hkh^{-1},h)}{\omega(g,h,k)\omega(ghkh^{-1}g^{-1},g,h)} (g\triangleright (h\triangleright v)), & g,h,k & \in\Lambda, & v &\in V_k,
\end{align*}
satisfying the following compatibility condition:
\begin{align*}
g\triangleright V_h & \subseteq V_{ghg^{-1}} & \text{for all } g,h & \in\Lambda.
\end{align*}
Morphisms in $\yd{\Bbbk^{\omega}\Lambda}$ are $\Lambda$-linear $\Lambda$-homogeneous maps. 
The tensor product of $V=\oplus_{g\in \Lambda}V$ and $W=\oplus_{g\in \Lambda}w$  is
$V\otimes W$ as vector space, with
\[(V\otimes W)_g=\bigoplus_{h\in G}V_h\otimes W_{h^{-1}g},\]
and for all $v\in V_g, w\in W_l$,
\[h\triangleright( v\otimes w)= \frac{\omega(hgh^{-1},hlh^{-1},l)\omega(h,g,l)}{\omega(hgh^{-1},h,l)}(h\triangleright v) \otimes (h\triangleright w).\]
The associativity constraints are the same as $\Vc_\Lambda^\omega$. The category is tensor braided, with braiding $c_{V,W}:V\ot W\to W\ot V$, $V,W\in\yd{\Bbbk^{\omega}\Lambda}$,
\begin{align*}
c_{V,W}(v\ot w) &= (g\triangleright w) \ot v, & g & \in\Lambda, & v &\in V_g, & w&\in W.
\end{align*}

\end{example}

\subsection{Bosonization for coquasi-Hopf algebras}
Now we recall the notation and results from \cite{ArdP-Jalg} but restricted to pointed coquasi-Hopf algebras.

Given a Hopf algebra $R$ in $\yd{\Bbbk^{\omega}\Lambda}$ with multiplication $\cdot:R\ot R\to R$ and comultiplication $\varDelta:R\to R\ot R$, $\varDelta(r)=r^{(1)}\ot r^{(2)}$, the \emph{bosonization} of $R$ by $\Bbbk^{\omega}\Lambda$ \cite[Definition 5.4]{ArdP-Jalg} is the coquasi-Hopf algebra $R\# \Bbbk^{\omega}\Lambda$ with underlying vector space $R\ot\Bbbk\Lambda$ and the following structure maps:
\begin{align*}
&(r\# g)(s\# h) = \frac{\omega(g,l,h) \omega(k,l,gh)}{\omega(k,g,lh) \omega(l,g,h)} r\cdot  (g\triangleright s)\# gh,
\\
&\Delta(r\# g) = \frac{1}{\omega(kj^{-1},j,g)} r^{(1)}\# lg \ \ot \ r^{(2)}\# g,
\\
&\Omega(r\#g, s\# h, t\# k)= \eps(r)\eps(s)\eps(t) \omega(g,h,k),
\end{align*}
for all $g,h,k,l\in\Lambda$, $r\in R_k$, $s\in R_l$, $t\in R$, where $r^{(1)}\ot r^{(2)}\in \oplus_j R_{kj^{-1}}\ot R_j$.

We have two canonical coquasi-Hopf algebra maps
\begin{align*}
\pi: & R\# \Bbbk^{\omega}\Lambda\to \Bbbk^{\omega}\Lambda, \ \pi(r\# g)=\eps(r)g, & \iota: &  \Bbbk^{\omega}\to R\# \Bbbk^{\omega}\Lambda, \ \iota(g)= 1\# g,
\end{align*}
such that $\pi\circ \iota=\id_{\Bbbk^{\omega}\Lambda}$.

Reciprocally, let $H$ be a coquasi-Hopf algebra and assume that there exist coquasi-Hopf algebra maps $\pi:H\to \Bbbk^{\omega}\Lambda$, $\iota:\Bbbk^{\omega}\Lambda\to H$ such that $\pi\circ \iota=\id_{\Bbbk^{\omega}\Lambda}$. Then $H\simeq R\# \Bbbk^{\omega}\Lambda$, where $R=H^{\co \pi}$ admits a structure of Hopf algebra in $\yd{\Bbbk^{\omega}\Lambda}$ \cite[Theorem 5.8]{ArdP-Jalg}.

\medspace

In particular this applies for $H=\oplus_{n\geq 0} H_n$ coradically graded such that $H_0=\Bbbk^{\omega}\Lambda$ \cite[6.1]{ArdP-Jalg}. Here, $R$ is a graded Hopf algebra in $\yd{\Bbbk^{\omega}\Lambda}$:
\begin{align*}
R&=\oplus_{n\geq 0} R_n, & \text{with } R_n&=R\cap H_n, \ n\geq 0, & \text{so }&R_0=\Bbbk 1.
\end{align*}

\subsection{Nichols algebras}
Nichols algebras can be defined over any abelian braided tensor category see \cite{Sch}.  In particular we may consider Nichols algebra over $\cC=\mathcal{Z}(^H\cM)$ or $\cC=\yd{H}$, where $H$ is a coquasi-bialgebra, see \cite{AS-pointed} for the definition when $H$ is a Hopf algebra and \cite{HLYYquasiHopf} for $H=\Bbbk^{\omega}\Lambda$.

Given an object $V\in \cC$ and $n\geq 3$, $V^{\ot n}$ denotes $(\cdots ((V\ot V)\ot \cdots )\ot V)$, $n$ copies of $V$. We consider the following (graded) Hopf algebras in $\cC$:
\begin{itemize}[leftmargin=*]
	\item the tensor algebra 
	$T(V)=\oplus_{n\geq 0} V^{\ot n}$, with product given by the canonical isomorphism $V^{\ot m} \ot V^{\ot n}\simeq V^{\ot (m+n)}$; the coproduct $\Delta:T(V)\to T(V)\ot T(V)$ is the unique graded algebra map such that $\Delta_{0,1}:V\to \Bbbk\ot V$ and $\Delta_{1,0}:V\to V\ot\Bbbk$ are the canonical isomorphisms.
	\item the tensor coalgebra 
	$C(V)=\oplus_{n\geq 0} V^{\ot n}$, with coproduct
	\begin{align*}
	\Delta &=\oplus_{m,n\geq 0}:C(V)\to C(V)\ot C(V), & \Delta_{m,n}:&V^{\ot (m+n)} \overset{\sim}{\to} V^{\ot m} \ot V^{\ot n}; 
	\end{align*}
	the product $\Delta:T(V)\to T(V)\ot T(V)$ is the unique graded coalgebra map induced by the canonical isomorphisms $\Bbbk\ot V \simeq V\simeq V\ot\Bbbk$.
\end{itemize}

\medspace

There exists a unique graded Hopf algebra map 
$T(V)\to C(V)$ in $\cC$, which is the identity on $V$. The \emph{Nichols algebra} $\cB(V)$ of $V$ is the image of this map: it is a graded Hopf algebra in $\cC$.

We may identify $\cB(V)$ as a quotient $\cB(V) = T(V)/\mathcal{J}(V)$ with the following universal property:
$\mathcal{J}(V)$ is the largest coideal of $T(V)$ spanned by elements of $\mathbb{N}$-degree
$\geq 2$. There are other characterizations of $\cB(V)$ \cite{Sch}.

A \emph{post-Nichols algebra} of $V$ is a graded Hopf subalgebra
$\cE=\oplus_{n\in\N_0} \cE_n$ of $C(V)$ in $\cC$ such that $\cE_1=V$, see \cite{AAR-Dividedpowers}; hence $\cB(V)\subseteq \cE$ and the set of primitive elements of $\cE$ is exactly $\cE_1$.

\section{Trivializations of elements in $H^3(\Lambda,\Bbbk^\times)$}\label{trivializacion}

In this section we study a family of 3-cocycles of finite abelian groups called \emph{trivializable}. These are the cocycles considered as associators for pointed tensor categories in Section \ref{section:basic radically graded}.

\medspace

Let $\Lambda$ be a finite abelian group. We denote by  $\wedge^n \Lambda$ the $n$-th exterior power of $\Lambda$, viewed as a $\mathbb{Z}$-module.
For each $\omega \in Z^3(\Lambda,\Bbbk^\times)$, Breen \cite[Proposition 4.1]{Breen} defined an alternating  trilinear map
\begin{align*}
\psi_\Lambda(\omega)(l_1,l_2 ,l_3) & =\prod_{\sigma \in \mathbb{S}_3} \omega(l_{\sigma(1)}, l_{\sigma(2)}, l_{\sigma(3)})^{\operatorname{sng}(\sigma)}, &  l_1,l_2,l_3 & \in \Lambda.
\end{align*}
The group homomorphism $\psi_\Lambda:Z^3(\Lambda,\Bbbk^\times)\to \operatorname{Hom}(\wedge^3 \Lambda,\Bbbk^\times)$ induces a group homomorphism 
\begin{equation*}
\psi_\Lambda:H^3(\Lambda,\Bbbk^\times)\to \operatorname{Hom}(\wedge^3 \Lambda,\Bbbk^\times).
\end{equation*}

Note that $\operatorname{Hom}(\Lambda^{\otimes 3},\Bbbk^\times)\subset Z^3(\Lambda,\Bbbk^\times)$. Hence, if $\Lambda$ is finite the restriction of $\psi_\Lambda$ to $\operatorname{Hom}(\Lambda^{\otimes 3},\Bbbk^\times)$ is surjective. Thus $\psi_\Lambda$ is surjective.

Given $\omega\in Z^3(\Lambda, \Bbbk^\times)$, we denote by $p^*\omega\in Z^3(\Gamma, \Bbbk^\times)$ the pull-back of $\omega$ by $p$; that is, the 3-cocycle defined by \begin{align*}
p^*\omega(g,h,k)&=\omega(p(g),p(h),p(k)), & & g,h,k \in \Gamma.
\end{align*}

\medspace


\begin{prop}\label{Prop trivializacion de 3-cociclos abeliano}
Let $\Lambda= \mathbb {Z}/n_1\mathbb{Z}\oplus \cdots \oplus \mathbb Z/n_m\mathbb{Z},$ and 
\[p:\Gamma:=\mathbb {Z}/2n_1\mathbb{Z}\oplus \cdots \oplus \mathbb Z/2n_m\mathbb{Z}\to \Lambda:=\mathbb {Z}/n_1\mathbb{Z}\oplus \cdots \oplus \mathbb Z/n_m\mathbb{Z},\] the canonical epimorphism. For any $(\omega,c)\in Z^3_{ab}(\Lambda,\Bbbk^\times)$ the pull back $p^*\omega \in H^3(\Gamma, \Bbbk^\times)$ is trivial.
\end{prop}
\begin{proof}
If $(\omega,c)\in Z^3_{ab}(\Lambda,\Bbbk^\times)$ the map 
\begin{align*}
    q&:\Lambda\to \Bbbk^\times,& q(l) &= c(l,l), & l&\in\Lambda,
\end{align*}
is a quadratic form on $\Lambda$; that is, $q(l^{-1})=q(l)$ for all $l\in\Lambda$ and the map 
\begin{align*}
    b_q(k,l)=q(kl)q(k)^{-1}q(l)^{-1},&& k,l\in \Lambda,
\end{align*}is a bicharacter.

The quadratic form $q$ determines completely the abelian cohomology class of the pair $(w,c)$, see \cite[Theorem 26.1]{EM2}. Using the map $q$, Quinn \cite{Quinn} defined an explicit  abelian 3-cocycle $(h,c)$ with $c(l,l)=q(l)$ for all $l \in \Lambda$. Assume that $\Lambda= \mathbb {Z}/n_1\mathbb{Z}\oplus \cdots \oplus \mathbb Z/n_m\mathbb{Z}$. For each $i\in \{1,\ldots ,m\}$ let $q_i:=q(\vec{e}_i)$ and $h_i\in Z^3(\mathbb {Z}/n_i\mathbb{Z},\Bbbk^\times)$ defined by 
\begin{align*} 
h_i(a,b,c) &=\begin{cases} 1, \qquad &\text{if } b+c<n_i,\\
q_i^{n_ia}. \quad  &\text{if $b+c\geq n_i$,}
\end{cases} & & 0\leq a,b,c<n_i.
\end{align*}
Then by \cite{Quinn} and \cite[Theorem 26.1]{EM2},  $h\in Z^3(\Lambda,\Bbbk^\times)$ given by $$h(\vec{x},\vec{y},\vec{z})=h_1(x_1,y_1,z_1)h_2(x_2,y_2,z_2)\cdots h_m(x_m,y_m,z_m),$$ is a 3-cocycle cohomologous to $\omega$. 

For any $n\in \mathbb{N}$ and  $a \in \Gamma$, we have $q(a^m)=q(a)^{2m^2}$. Hence $h_i$ has order at most two. By Example \ref{Ejemplo grupo ciclico}, via the epimorphism \[p:\mathbb {Z}/2n_1\mathbb{Z}\oplus \cdots \oplus \mathbb Z/2n_m\mathbb{Z}\to \mathbb {Z}/n_1\mathbb{Z}\oplus \cdots \oplus \mathbb Z/n_m\mathbb{Z},\] the pull-back $p^*h$ is trivial, and also $p^*\omega$.
\end{proof}

\begin{theorem}\label{Condiciones equivalentes}
Let $\omega \in H^3(\Lambda,\Bbbk^\times)$. The following statements are equivalent:
\begin{enumerate}[leftmargin=*,label=\rm{(\alph*)}]
\item\label{item:psi-trivial} $\psi_\Lambda(\omega)\equiv1$.
\item\label{item:braided-fusion-pointed} The braided fusion category $\yd{\Bbbk^{\omega}\Lambda}$ is pointed.
\item\label{item:omega-trivializable} There exists a finite abelian group $\Gamma$ and a group epimorphism $p:\Gamma\to \Lambda$  such that the pullback $p^*\omega \in H^n(\Gamma,\Bbbk^\times)$ is trivial.
\end{enumerate}
\end{theorem}
\begin{proof}
\ref{item:psi-trivial} $\iff$ \ref{item:braided-fusion-pointed}. For each $l\in \Lambda$, the map \begin{align*}
    \beta_l:\Lambda\times \Lambda\to \Bbbk^\times, && \beta_l(g,h)=\frac{\omega(g,l,h)}{\omega(g,h,l) \omega(l,g,h)}
\end{align*} is a 2-cocycle; that is, it satisfies the equation
\begin{align*}
\beta_l(g,h)\beta_l(gh,k)&=\beta_l(g,hk)\beta_l(h,k), &
\text{for all } & g,h,k\in \Lambda.
\end{align*}
By \cite[Example 6.3]{GP} we have  an exact sequence of groups
\[0\to \widehat{\Lambda} \to \operatorname{Inv}(\yd{\Bbbk^{\omega}\Lambda}) \to \Lambda_{\omega}\to 0,\]
where $\Lambda_\omega=\{l\in \Lambda: 0=[\beta_l]\in H^2(\Lambda,\Bbbk^\times)\}$. 
Then $\yd{\Bbbk^{\omega}\Lambda}$ is pointed if and only if $0=[\beta_l]$ for all $l\in \Lambda$.

Since $\Bbbk^\times$ is divisible, $\beta_l$ has trivial cohomology class if and only if $\beta_l$ is symmetric. In conclusion, $\yd{\Bbbk^{\omega}\Lambda}$ is pointed if and only if $\beta_l(g,h)=\beta_l(h,g)$ for all $l,g,h\in \Lambda$. Since
\begin{align*}
\frac{\beta_{l}(g,h)}{\beta_l(h,g)}&=\psi_\Lambda(l,g,h) & \text{for all }& g,h,l\in\\Lambda,
\end{align*}
$\yd{\Bbbk^{\omega}\Lambda}$ is pointed if and only if $\psi_\Lambda(\omega)=1$. 

\ref{item:omega-trivializable} $\implies$ \ref{item:psi-trivial}. Let $p:\Gamma\to \Lambda$ be an epimorphism of finite abelian groups. By \cite[\S 7.2, Proposition 3 ]{Bourbaki}, the map 
\begin{align*}
\wedge^{n}(p): \wedge^{n} \Gamma &\to \wedge^{n}\Lambda, & 
g_1\wedge\cdots \wedge g_n &\mapsto p(g_1)\wedge \cdots \wedge p(g_n),
\end{align*}
is surjective. Since $\Lambda$ is finite, the group homomorphism 

\begin{align*}
\wedge^n(p)^*: \operatorname{Hom}(\wedge^n \Lambda,\Bbbk^\times) &\to \operatorname{Hom}(\wedge^n \Gamma,\Bbbk^\times)\\
f &\mapsto [g_1\wedge \cdots \wedge g_n\mapsto f(p(g_1)\wedge \cdots \wedge p(g_n))].
\end{align*}is injective for all $n$.

Let $\omega\in H^3(\Lambda,\Bbbk^\times)$ such that $p^*(\omega)=0$. Then $\wedge^{(3)} (p)^*\circ \psi_\Lambda (\omega)=0$, since the diagram 
\begin{equation}
\begin{tikzcd}
H^3(\Lambda,\Bbbk^\times) \ar{d}{\psi_\Lambda} \ar{r}{p^*} & H^3(\Gamma,\Bbbk^\times) \ar{d}{\psi_\Gamma}\\
\operatorname{Hom}(\wedge^3 \Lambda,\Bbbk^\times) \ar{r}{\wedge^{(3)} (p)^*}& \operatorname{Hom}(\wedge^3 \Gamma,\Bbbk^\times)    
\end{tikzcd}
\end{equation}
is commutative. By the injectivity of $\wedge^3(p)^*$, we have that $\psi_\Lambda (\omega)=0$.

\ref{item:braided-fusion-pointed} $\implies$ \ref{item:omega-trivializable}. Assume that \ref{item:braided-fusion-pointed} holds. Then  there is a finite abelian group $\Gamma$ and an abelian 3-cocycle $(\alpha,c)\in Z^3_{ab}(\Gamma,\Bbbk^\times)$ such that $\yd{\Bbbk^{\omega}\Lambda} \cong \Vc_\Gamma^{(\alpha,c)}$ as braided fusion categories. The forgetful functor 
$\yd{\Bbbk^{\omega}\Lambda}\to \Vc_\Lambda^{\omega}$ defines a group epimorphism $\pi_1:\Gamma\to \Lambda$ such that $\pi_1^*([\omega])=[\alpha]$. By Proposition  \ref{Prop trivializacion de 3-cociclos abeliano}, there exists an abelian group $\Gamma_2$ and an epimorphism $\pi_2:\Gamma_2\to \Gamma_1$ such that $\pi_1^*([\alpha])=0$, hence $\pi_2\circ \pi_1:\Gamma_2\to \Lambda$ trivializes $\omega$.
\end{proof}

\begin{definition}
Let $\omega\in H^3(\Lambda,\Bbbk^\times)$. We say that $\omega$ is \emph{trivializable} if it satisfies one of the equivalent conditions of Theorem \ref{Condiciones equivalentes}. If $p:\Gamma\to \Lambda$  is an epimorphism of abelian groups such that the pullback $p^*\omega \in H^3(\Gamma,\Bbbk^\times)$ is trivial, we say that $\omega$ is $p$-\emph{trivial}.  
\end{definition}

\begin{example}\label{Ejemplo grupo ciclico}
Let $C_n$ be the cyclic group of order $n$ generated by $\sigma$. Then
\begin{align*}
\xymatrix{\cdots \ar@{->}[r]^{N}& \Z C_n \ar@{->}[r]^{\sigma-1}& \Z C_n \ar@{->}[r]^{N}& \Z C_n \ar@{->}[r]^{\sigma-1}& \Z C_n \ar@{->}[r]& \Z}
\end{align*}
where $N= 1+ \sigma + \sigma^2 +\cdots +\sigma ^{n-1}$ is a free resolution of $\mathbb Z$. Thus,
\[ H^3(C_n, \Bbbk^\times)=\mathbb{G}_m(n):=\{a\in \Bbbk^\times: a^n=1\}.\]
Let $m, n \in \mathbb{N}$ such that $n|m$ and  $\pi: C_{m}\to C_{n}$ be the canonical group epimorphism. The induced map is
\begin{align*}
\pi^*: & H^3(C_n,\Bbbk^\times)\to H^3(C_m,\Bbbk^\times), &
q\mapsto q^{\frac{m}{n}}.
\end{align*}
Hence, if $q\in H^3(C_n, \Bbbk^\times)$ has order $s$, the canonical epimorphism $\pi: C_{sn}\to C_n$ trivializes $q$. Thus $\pi:C_{n^2}\to C_n$ trivializes all elements in $H^3(C_n, \Bbbk^\times)$. 
\end{example}

\begin{example}
Let $\zeta$ be a $n$-th root of unity, $\Lambda=(\mathbb{Z}/n\mathbb{Z})^{\oplus 3}$. We define
\begin{align*}
\omega& \in Z^3(\Lambda,\Bbbk^\times), & 
\omega(\vec{x},\vec{y},\vec{z})&=\zeta^{x_1y_2z_3}, & 
\vec{x},\vec{y},\vec{z}&\in \Lambda.
\end{align*}
Then, $\psi_\Lambda(\omega)(\vec{x},\vec{y},\vec{z})= \zeta^{\operatorname{det}([\vec{x},\vec{y},\vec{z}])}$, so $\psi_\Lambda(\omega)\neq 0$ and $\langle \psi(\omega) \rangle = \operatorname{Hom}(\wedge^3 \Lambda,\Bbbk^\times)$. It follows by Theorem \ref{Condiciones equivalentes} that $\omega$ is not trivializable.
\end{example}

\begin{remark}
In the proof of Proposition \ref{Prop trivializacion de 3-cociclos abeliano}, we prove that every cohomology class of a 3-cocycle $\omega$ that admits an abelian structure $(\omega,c)$ has order two. If $\Lambda$ is a cyclic group of odd order and $\omega \in H^3(\Lambda,\Bbbk^\times)$ is a non-zero element, there is not $c\in C^2(\Lambda,\Bbbk^\times)$ such that $(\omega,c)\in Z^3_{ab}(\Lambda,\Bbbk^\times)$, however $\psi_\Lambda(\omega)=0$.
\end{remark}

\section{Pointed coradically graded coquasi-Hopf algebras}\label{section:basic radically graded}

Let $\Gamma$ and $\Lambda$ be abelian groups and $p:\Gamma\to \Lambda$ a group epimorphism. We fix a section $\iota:\Lambda\to\Gamma$ of $p$, and a 3-cocycle $\omega \in H^3(\Lambda,\Bbbk^{\times}$.

We assume that there is $\alpha : \Gamma \times \Gamma \to \Bbbk^\times,$ such 
that $\delta(\alpha)=p^*\omega$; that is,
\begin{align*}
p^*\omega(g,h,k)&=\frac{\alpha(g,h)\alpha(gh,k)}{\alpha(g,hk)\alpha(h,k)}, & g,h,k&\in\Gamma,
\end{align*}




\subsection{Trivializing the non-associativity of Nichols algebras}

We consider the functor $\yd{\Bbbk^\omega\Lambda}\to \yd{\Bbbk^{p^*\omega}\Gamma}$ given on the objects by
\begin{align*}
V\mapsto&\widehat{V}, & \text{with }&\Gamma\text{-grading} & \widehat{V}_g &=\begin{cases}
V_{k} & g=\iota(k), \\ 0 & g\notin \iota(\Lambda),
\end{cases}
\end{align*}
and $\Gamma$-action via $p$; on the morphisms, it is just the identity. 

As $\delta(\alpha)=p^*\omega$, there is a braided tensor equivalence 
\begin{align*}
(F_\alpha,\overline{\alpha}):\yd{\Bbbk^{p^*\omega}\Gamma} \to \yd{\Bbbk\Gamma},
\end{align*}
where $F_\alpha(V)=V$ as $\Gamma$-graded vector spaces, with $\Gamma$-action 
\begin{align*}
g \cdot v&=\frac{\alpha(h,g)}{\alpha(g,h)}g\rhd v, & g,h & \in\Gamma, & v &\in V_g;
\end{align*} 
the functor is the identity for morphisms; the isomorphism constraints are
\begin{align*}
\overline{\alpha}_{V,V'}:F_\alpha(V\otimes V')&\to F_\alpha(V)\otimes F_\alpha(V)\\
v\otimes v' &\mapsto \alpha(g,h)\ v\otimes v', & g,h & \in\Gamma, \ v \in V_g, \ v'\in V_h.
\end{align*}

Fix $V\in \yd{\Bbbk^\omega\Lambda}$. Hence $W:=F_{\alpha}(\widehat{V}) \in \yd{\Bbbk\Gamma}$ is a braided vector space of diagonal type: there exists a basis $(x_i)_{i\in \I}$, elements $g_i\in\Gamma$, $\chi_i\in\widehat{\Gamma}$ such that $x_i\in W_{g_i}^{\chi_i}$, so the braiding is
\begin{align*}
c(x_i\ot x_j)& =g_i\cdot x_j\ot x_i= q_{ij} \ x_j \ot x_i, & q_{ij}&:=\chi_j(g_i), & i,j &\in\I.
\end{align*}
Coming back to $V$, let $\ell_i=p(g_i)\in\Lambda$, $i\in\I$. As $V=W$ as vector spaces and the $\Gamma$-grading on $W$ is induced by $\iota$, we have that $g_i=\iota(\ell_i)$ and $x_i\in V_{\ell_i}$ for all $i\in\I$. The \emph{quasi}-braiding in $\yd{\Bbbk^\omega\Lambda}$ is given by
\begin{align*}
c_V(x_i\ot x_j)& =g_i\triangleright x_j\ot x_i= q_{ij} \frac{\alpha(\ell_i,\ell_j)}{\alpha(\ell_j,\ell_i)} \ x_j \ot x_i, & i,j &\in\I.
\end{align*}

\subsection{Generation in degree one}\label{subsec:generation-deg-one}

We recall some results about the FRT construction.
Let $H(W)$ be the bialgebra corresponding to the diagonal braided vectors space $(W,c)$ \cite[VIII.6]{ka}. As an algebra, $H(W)$ is presented by generators $T_{j}^i$, $i,j\in\I$ and relations
\begin{align*}
& q_{ij}T_{j}^nT_i^m -q_{nm}T_i^m T_{j}^n, & i,j,m,n & \in\I.
\end{align*}
Hence $H(W)$ is a quantum linear space, so in particular it is $\Z^\I$-graded, with $\deg T_i^j=\alpha_i$, $i,j\in\I$. The coproduct satisfies
\begin{align*}
\Delta(T_i^j)&=\sum_{k\in \I} T_i^k\ot T_k^j, & i,j &\in\I.
\end{align*}
Hence $W$ is an $H(W)$-comodule with coaction
\begin{align*}
\rho: & W\to H(W)\ot W, & \rho(x_i) &= \sum_{j\in\I} T_i^j \ot x_j, & i&\in\I.
\end{align*}
The $R$-matrix $\mathtt{r}:H(W)\ot H(W)\to \Bbbk$ is determined by
\begin{align*}
\mathtt{r}(T_i^m \ot T_{j}^n) &= q_{ji} \delta_{i,m} \delta_{j,n}, & i,j,m,n & \in\I.
\end{align*}
and $c$ is also the braiding in the category of $H(W)$-comodules.

\begin{theorem}\label{teo:pre-Nichols-are-Nichols}
Let $R=\oplus_{n\geq 0} R_n\in \yd{\Bbbk^\omega\Lambda}$ be a post-Nichols algebra of $V=R_1$ such that $\dim R<\infty$. Then $R=\cB(V)$. 
\end{theorem}
\pf
By abuse of notation, let $\alpha: H(W)\ot H(W)\to \Bbbk$,
\begin{multline*}
\alpha(T_{i_1}^{m_1} \dots T_{i_s}^{m_s}, T_{j_1}^{n_1} \dots T_{j_t}^{n_t}) = \delta_{i_1,m_1} \dots \delta_{i_s,m_s} \delta_{j_1,n_1} \dots \delta_{j_t,n_t} \\ 
\alpha(g_{i_1}\dots g_{i_s}, g_{j_1} \dots g_{j_t}), \qquad s,t\in\N, \ i_k,m_k,j_l,n_l\in\I.
\end{multline*}
As $H(W)$ is $\Z^{\I}$-graded, the map is well-defined, and $\alpha(1,x)=\alpha(x,1)=\eps(x)$ for all $x\in H(W)$. Hence we may consider the coquasi-bialgebra $H(W)^\alpha$ obtained by a 2-cocycle deformation by $\alpha$.

Notice that $(V,c_V)$ is the image of $(W,c)$ under the braided equivalence $^{H(W)}\cM\to\ ^{H(W)^\alpha} \cM$ induced by the 2-cocycle $\alpha$, and this equivalence takes post-Nichols algebras of $(W,c)$ to post-Nichols algebras of $(V,c_V)$. Hence $R$ is the image of a post-Nichols algebra $R'$ of $(W,c)$, which is of diagonal type. By \cite{Ang}, $R'=\cB(W)$, so $R=\cB(V)$.
\epf

\subsection{Pointed coquasi-Hopf algebras and de-equivariantization}\label{subsec:de-equiv-pointed-graded}

Let $H$ be a coquasi-Hopf algebra and $G$ be an affine group scheme over $\Bbbk$. A central inclusion of $G$ in $H$ is a full braided embedding $\iota : \operatorname{Rep}(G) \to  \mathcal{Z}(\ ^H\mathcal{M})$ such that the composition $\iota\circ U:\Rep(G) \to  
\ ^H\mathcal{M}$ is full, where $U:\mathcal{Z}(
\ ^H\mathcal{M})\to\ ^H\mathcal{M}$ is the forgetful functor.

Let $\mathcal{O}(G)$ be the algebra of regular function over $G$. The algebra $\mathcal{O}(G)$ is a commutative algebra in the symmetric category $\operatorname{Rep}(G)$, and thus a commutative algebra in the braided tensor category $\mathcal{Z}(
\ ^H\mathcal{M})$. Following \cite{ENO3}, we define  the de-equivariantization $^H\mathcal{M}(G)$ of $^H\mathcal{M}$ by $G$, as the monoidal category of
of left $\mathcal{O}(G)$-modules in $^H\mathcal{M}$, with the tensor product $M\otimes_{\mathcal{O}(G)} N$.

Now we prove that each coradically graded pointed coquasi-Hopf algebra with trivializable 3-cocycle comes from the construction above.

\begin{theorem}\label{thm:A-corad-graded-deequiv}
Let $A$ be a finite-dimensional coradically graded coquasi-Hopf algebra such that $A_0\simeq \Bbbk^{\omega}\Lambda$, where $\omega$ is trivializable. Then $^{A}\cM$ is a de-equivariantization of a coradically graded pointed Hopf algebra over an abelian group.

\end{theorem}
\pf
By \cite{ArdP-Jalg} there exists a post-Nichols algebra $R=\oplus_{n\geq 0} R_n\in \yd{\Bbbk^\omega\Lambda}$ of $V=R_1$ such that $A\simeq R\# \Bbbk^{\omega}\Lambda$; hence $\dim R<\infty$, and by Theorem \ref{teo:pre-Nichols-are-Nichols}, $R=\cB(V)$.
We consider $\widehat{V}\in \yd{\Bbbk^{p^*\omega} \Gamma}$: as the braiding is the same, $\cB(V)\simeq \cB(\widehat{V})$ as braided Hopf algebras, and 
\begin{align}\label{mapa de coquasi}
 \pi:=(\id\ot p): B:=\cB(\widehat{V})\# \Bbbk^{p^*\omega} \Gamma \to A=\cB(V)\#\Bbbk^{\omega}\Lambda
\end{align}  
is a projection of coquasi-Hopf algebras. 

Given an epimorphism $f:H\to Q$ of finite dimensional coquasi-Hopf algebras, it follows by \cite[Proposition 5.1]{BN} that $$H^{\operatorname{co}f}:=\{b\in B: (\operatorname{id}\otimes f)\Delta (b)=b\otimes 1\},$$ admits a structure of commutative algebra in $\mathcal{Z}(^H\cM)$ such that the tensor category of left $H^{\operatorname{co}f}$-modules in $\ ^H\cM$ is tensor equivalent to $^Q\cM$.

Set $K=\ker p$. We claim that there is a central inclusion $\iota:\ \operatorname{Rep}(\widehat{K})\to \mathcal{Z}(^{B}\mathcal{M})$, such that the central algebra $\mathcal{O}(\widehat{K})=\Bbbk K$ is the central algebra associated to the epimorphism \eqref{mapa de coquasi}. 

The inclusion $\Bbbk K\xhookrightarrow{} B, a\mapsto 1\#a$, is an injective  coquasi-Hopf algebra morphism, that induces a full tensor embedding $$\operatorname{Rep}(\widehat{K})=\Vc_{K}\xhookrightarrow{}  ^B\mathcal{M}.$$ 
Let  $V_a= \Bbbk v\in \  ^{\Bbbk K}\mathcal{M}$ be a one-dimensional comodule with $\rho(v)=a\otimes v$, $a\in K$. As $1\#a$ is a central group-like of $B$, for any $M\in\ ^{B}\mathcal{M}$, the flip map
\begin{align*}
c_{M,V_a}:M\otimes V_a&\to V_a\otimes M &
 m\otimes v &\mapsto v\otimes m,
\end{align*}
is an  isomorphism of $B$-comodules. Now \eqref{equ: centro} follows from the fact that 
\begin{align*}
p^*\omega(a,g,h)&=p^*\omega(g,a,h)=p^*\omega(g,h,a)=1 & 
\text{for all } & g,h \in \Gamma, a\in K.    
\end{align*}

Since $\Bbbk K=B^{\operatorname{co}\pi}:=\{b\in B: \operatorname{id}\otimes \pi\Delta (b)=b\otimes 1\}$, the central algebra associated to the surjective tensor functor  $\pi_*: \ ^B\cM \to \ ^A\cM$ is exactly $\Bbbk K$. By \cite[Proposition 5.1]{BN}, $^A\cM$ is a de-equivariantization of $^B\cM$ by $\widehat{K}$.

Recall that there is a map $\alpha:\Gamma\times \Gamma \to \Bbbk^\times$ such that $\delta(\alpha)=p^*\omega$; we can extend $\alpha$ linearly to a map $\alpha:B \ot B\to \Bbbk$ such that 
\begin{align*}
\alpha(B_n\ot B) & =\alpha(B\ot B_n)=0 & \text{for all }& n>0, 
\end{align*}
and $\alpha$ is a twist. Hence $H:=B^{\alpha}$ is a Hopf algebra; as a coalgebra, $H=B$ is coradically graded, with $H_0=\Bbbk\Gamma$. Hence $H\simeq R'\#\Bbbk\Gamma$ for some graded Hopf algebra $R'\in\yd{\Bbbk\Gamma}$, where $R'_1=F_{\alpha}(\widehat{V})$, so $H\simeq \cB(W)\# \Bbbk \Gamma$. Since $^H\cM$ is tensor equivalent to $^B\cM$, we have that $^A\cM$ is a de-equivariantization of the Hopf algebra $H$ by the group $\widehat{K}$.
\epf

\begin{proof}[Proof of Theorem \ref{main result}]

Let $\cC$ be a de-equivariantization of $^H \cM$, where $H$ is a finite-dimensional pointed Hopf algebra with abelian coradical. By \cite{angiono2016liftings} we may assume that $H$ is coradically graded. By \cite[Proposition 3.3]{AGP}, $\cC$ is the category of comodules over a finite-dimensional coradically graded coquasi-Hopf algebra with trivializable 3-cocycle.

On the other hand, assume that $\cC$ is a coradically graded tensor category such that $G(\cC)$ is abelian and $\psi_{G(\cC)}(\omega_{\cC})\equiv 1$; by hipothesis, $\cC$ is the category of comodules over a finite-dimensional coradically graded coalgebra. As $\cC$ is pointed, this coalgebra has an structure of pointed coquasi-Hopf algebra, see \cite[Proposition 2.6]{finite-tensor}, and Theorem \ref{thm:A-corad-graded-deequiv} applies.
\end{proof}

\begin{coro}
Let $\cC$ be a coradically graded finite tensor category admitting a braiding. Then $\cC$ can be realized as  the de-equivariantization of the category of comodules of a finite dimensional  coradically graded pointed  Hopf algebra over an abe\-lian group. \end{coro}
\begin{proof}
Since $\cC$ is braided, the full subcategory generated by the invertible objects is also braided. Thus, the associator of $\cC$ is an abelian 3-cocycle. By Proposition \ref{Prop trivializacion de 3-cociclos abeliano} and Theorem \ref{main result}, $\cC$ is tensor equivalent to the de-equivariantization of a Hopf algebra of the form
$\cB(V)\#\Bbbk\Gamma$.  Hence $(\cC,c)$ is realized as the comodules on a coquasitriangular coquasi-Hopf algebra of the form $(\cB(V)\#\Bbbk^\omega \Lambda,r)$.
\end{proof}

Let $(\cB,c)$ be a braided tensor category. The Mueger's center or symmetric center $\mathcal{Z}_2(\cB)$ is the full tensor subcategory of all objects $Y\in \cB$ such that $c_{X,Y}\circ c_{Y,X}=\operatorname{id}_{Y\otimes X}$ for all $X\in\cB$. A braided tensor category is called symmetric if $\cB=\mathcal{Z}_2(\cB)$.

\begin{example}
Let $\cB=\Vc_{\Lambda}^{(\omega,c)}$ be a pointed fusion category. The Mueger's center of $\cB$ is the fusion subcategory with simple objects  given by 
\[\mathcal{Z}_2(\Lambda):=\{g\in \Lambda: c(g,h)c(h,g)=1, \ \text{for all }  h \in \Lambda\}.\]

Recall that the  map $q:\Lambda\to \Bbbk^\times, g\mapsto c(g,g)$ is a quadratic form. Since $q(gh) q(g)^{-1} q(h)^{-1} =c(g,h)c(h,g)$, we have that  \[\mathcal{Z}_2(\Lambda):=\{g\in \Lambda: q(gh)=q(g)q(h), \ \text{for all }  h \in \Lambda\}.\] Since $q$ only depends on the abelian cohomology class of $(\omega,c),$ the same is true of $\mathcal{Z}_2(\Lambda)$.
The restriction $q_{|\mathcal{Z}_2(\Lambda)}$ is a group homomorphism in $\operatorname{Hom}(\mathcal{Z}_2(\Lambda),\mathbb{Z}/2\mathbb{Z})$. Hence  $\mathcal{Z}_2(\cB)$ has a (possibly trivial) $\mathbb{Z}/2\mathbb{Z}$-grading.
\end{example}

Let $(\cB(V)\#\Bbbk^\omega[\Lambda],r)$ be a coquasitriangular coquasi-Hopf algebra. The $r$-form defines an abelian structure  $(\omega,c)\in Z^3(\Lambda,\Bbbk^\times)$, where $c:=r_{G(H)\times G(H)}$. Assume without loss of generality that $\omega|_{\mathcal{Z}_2(\Lambda)^{\times 3}}\equiv 1$.

As in \cite[Section 4]{bontea2017pointed}, the condition \eqref{r-matrix 1} implies that $V\in \mathcal{Z}_2(\Vc_{\Lambda}^{(\omega,c)})_-.$  Hence $\cB(V)\#\Bbbk\mathcal{Z}_2(\Lambda)$ is a coquasitriangular Hopf subalgebra of $\cB(V)\#\Bbbk^\omega\Lambda$. Then \cite[Theorem 1.1]{bontea2017pointed}, can be applied to $\cB(V)\#\Bbbk \mathcal{Z}_2(\Lambda)$. Hence, $r_1:=r|_{V\ot V}$ is a morphism in $\mathcal{Z}_2(\Vc_{\Lambda}^{(\omega,c)})$. 

In conclusion, as in \cite{bontea2017pointed}, every coquasitriangular structure on a coradially graded coquasi-Hopf algebra is determined by the following data: an abelian 3-cocycle structure $(\omega, c)\in Z_{ab}^3(\Lambda,\Bbbk^\times)$, an object $V\in \mathcal{Z}_2(\Vc_{\Lambda}^{(\omega,c)})_-,$ and  a morphism  $r_1:=r|_{V\ot V}$  in $\mathcal{Z}_2(\Vc_{\Lambda}^{(\omega,c)})$ .


\subsection{A pointed coquasi-Hopf algebra over an abelian group  with non-trivializable associator}

Let $A$ and  $B$  be finite abelian groups and $\alpha \in Z^2(A,\widehat{B})$ a 2-cocycle, that is, a  map $\alpha: A\times A\to \widehat{B}$ such that 
\begin{align*}
\alpha(x,y)\alpha(xy,z)=\alpha(x,yz)\alpha(y,z),&& x,y,z \in A.
\end{align*} 
We denote by $\widehat{B} \rtimes_{\alpha} A$ the central extension of $A$ by $\widehat{B}$ associated to $\alpha$. Explicitly, $\widehat{B} \rtimes_{\alpha} A=\widehat{B} \times A$ as a set, and the product is given by
\begin{align*}
(x,a)(x',a')&=(xx'\alpha(a,a'),aa'), & a,a'\in A, \ & x,x'\in B.
\end{align*}

The function 
\begin{align*}
\omega_\alpha((x_1,a_1),(x_2,a_2),(x_3,a_3))=\alpha(a_1,a_2)(x_3),
\end{align*}
is a 3-cocycle $\omega_\alpha\in Z^3(A\oplus B, \Bbbk^\times)$. It is easy to see that $\psi_{A\oplus B}(\omega_\alpha)=0$ if and only if $\alpha(a_1,a_2)=\alpha(a_2,a_1)$  for all $a_1, a_2\in A$.
By \cite[Theorem 3.6]{uribe}, the braided  categories  $\yd{\Bbbk^\omega_\alpha A\oplus B}$  and $\yd{\Bbbk \widehat{B} \rtimes_{\alpha} A}$  are equivalent.

\begin{example}
Let $A=\mathbb{Z}/2\mathbb{Z}\oplus \mathbb{Z}/2\mathbb{Z}$ and $B=\mathbb{Z}/2\mathbb{Z}$.  We define
\begin{align*}
\alpha: A\times A\to \widehat{B}, && \alpha((m_1,m_2),(n_1,n_2))=\chi^{m_1n_2},
\end{align*}
where $\chi:\mathbb{Z}/2\mathbb{Z}\to \Bbbk^\times$ is the non-trivial character.

The associated 3-cocycle is \[ \omega_\alpha(\vec{x},\vec{y},\vec{z})=(-1)^{x_1y_2z_3},\]where  $\vec{x},\vec{y},\vec{z}\in (\mathbb{Z}/2\mathbb{Z})^{\oplus 3}$. Then, $$\psi_\Lambda(\omega_\alpha)(\vec{x},\vec{y},\vec{z})= (-1)^{\operatorname{det}([\vec{x},\vec{y},\vec{z}])}.$$ Thus, $\psi(\omega)\neq 0$. It follows by Theorem \ref{Condiciones equivalentes} that $\omega$ is not trivializable.

The group $\widehat{B} \rtimes_{\alpha} A$ is isomorphic to $D_4$, the dihedral group of order $8$: it is a non-abelian group of order eight with two elements of order four.

In \cite[Example 6.5]{MH99}, Milinski and Schneider constructed a Nichols algebra $\cB(V)$ of dimension 64 over $D_4$. Since the braided categories $\yd{\Bbbk^\omega_\alpha A\oplus B}$  and $\yd{\Bbbk \widehat{B} \rtimes_{\alpha} A}$  are equivalent, there is a Nichols algebra $\cB(V')$ of dimension 64 in $\yd{\Bbbk^\omega_\alpha A\oplus B}$ . Hence, the bosonization
$A:= \cB(V')\# \Bbbk^{\omega_\alpha}(\mathbb{Z}/2\mathbb{Z})^{\oplus 3}$ is a coradically graded coquasi-Hopf algebra with non-trivializable associator.
\end{example}

\end{document}